\documentclass[11pt]{article}

\usepackage{amssymb,amsmath,amsthm}
\usepackage{amsrefs}
\usepackage{cancel}

\newcommand{\from}{\colon}

\newcommand{\E}{\mathrel{E}}
\renewcommand{\P}{\mathbb{P}}
\newcommand{\PP}{\textsc{\textbf{P}}}
\newcommand{\NP}{\textsc{\textbf{NP}}}

\newcommand{\ELEMENTARY}{\textsc{\textbf{ELEMENTARY}}}
\newcommand{\SPACE}{\textsc{\textbf{SPACE}}}

\newcommand{\PSPACE}{\textsc{\textbf{PSPACE}}}
\newcommand{\TIME}{\textsc{\textbf{TIME}}}

\newcommand{\powerset}{\mathcal{P}}
\newcommand{\pow}{\powerset}

\newcommand{\strings}{2^{< \omega}}
\newcommand{\langs}{\pow(\strings)}
\newcommand{\concat}{{}^{\smallfrown}}
\newcommand{\str}[1]{2^{#1}}

\renewcommand{\subset}{\subseteq}

\newcommand{\<}{\langle}
\renewcommand{\>}{\rangle}

\newcommand{\comp}[1]{{#1}^{\text{c}}} 

\newcommand{\restrict}{\restriction}
\newcommand{\cantor}{2^\omega}

\newcommand{\define}[1]{\emph{#1}}

\newcommand{\join}{\oplus}
\newcommand{\union}{\cup}

\newcommand{\biginters}{\bigcap}

\newcommand{\N}{\mathbb{N}}
\newcommand{\F}{\mathbb{F}}

\theoremstyle{plain}
\newtheorem{thm}{Theorem}[section]

\newtheorem{prop}[thm]{Proposition}
\newtheorem{cor}[thm]{Corollary}

\theoremstyle{definition}
\newtheorem{defn}[thm]{Definition}

\begin{document}

\author{Andrew Marks\thanks{The author was partially
supported by NSF grants DMS-1204907
and DMS-1500974 and the John Templeton Foundation under Award No.\ 15619.
This work forms part of the author's PhD thesis which was done under the
excellent supervision of Ted Slaman. The author would also like to thank
Jan Reimann and Alekos Kechris for helpful conversations, and the anonymous 
referees for many helpful comments and corrections}
}
\title{The universality of polynomial time Turing equivalence}

\date{\today}

\maketitle

\begin{abstract}
We show that polynomial time Turing equivalence and a large class of other equivalence relations from computational complexity theory are universal countable Borel equivalence relations. We then discuss ultrafilters on the invariant Borel sets of these equivalence relations which are related to Martin's ultrafilter on the Turing degrees.
\end{abstract}

\section{Introduction}

\subsection{Universal resource bounded equivalence relations from
computability}

In this paper, we study the global complexity of resource bounded 
reducibilities from computability theory. 
If $\mathcal{C}$ is
a subset of the Turing reductions such as the polynomial time Turing reductions,
and $x, y \in \langs$ are languages, then we write $x \leq_T^\mathcal{C} y$ if
$x$ is Turing reducible to $y$ via a Turing reduction in $\mathcal{C}$.
Similarly, if $\mathcal{C}$ is a subset of the many-one reductions, then we
analogously write $x \leq_m^\mathcal{C} y$ if $x$ is many-one
reducible to $y$ via a many-one reduction from $\mathcal{C}$.
Now these relations are not transitive for arbitrary $\mathcal{C}$. However, we will be
particularly interested in the special case where we fix a function $g
\from \N \to \N$ which is strictly increasing, time constructible, and
$g(n) \geq n^2$, and let $\mathcal{C}$ be the union over $k$ of all
reductions computable in $O(g^k)$-time or $O(g^k)$-space. Note that here we
use $g^k$ to denote the $k$-fold composition of $g$, $g^k = \underbrace{g
\circ \ldots \circ g}_{\text{k times}}$, and not its $k$-fold product. 
In this case, since there is only a polynomial amount of overhead required
to simulate the composition of two reductions, it is easy to see that the
associated reducibilities are transitive and symmetrize to equivalence
relations which we study below.

Previous investigations of the global properties of resource bounded 
reducibilities in computability have 
focused mainly on the
theory of these structures \cites{Ladner1975, MR1079470, MR1785026,
MR1255600}. We take a different approach, using the framework of
Borel reducibility between Borel equivalence relations. 
A \define{Borel equivalence relation} $E$ on a standard Borel space $X$ is an
equivalence relation that is a Borel subset of $X \times X$.
If $E$ and $F$ are both Borel
equivalence relations on $X$ and $Y$, then $E$ is \define{Borel
reducible} to $F$, noted $E \leq_B F$, if there is a Borel function $f
\from X \to Y$ such that for all $x, y \in X$, we have $x \mathrel{E} y
\iff f(x) \mathrel{F} f(y)$. Such a function induces an injection $\hat{f}
\from X/E \to Y/F$, and we can view Borel reducibility as comparing the
difficulty of classifying $E$ and $F$ by invariants, where if $E \leq_B F$,
then any complete set of invariants for $F$ can be used as a set of
complete invariants for $E$. 
Borel reducibility is also sometimes viewed as
describing \define{Borel cardinality}, where the injection $\hat{f}$ is a
Borel witness that the quotient of $E$ injects into the quotient of $F$.

A Borel equivalence relation is said to be \define{countable} if all of its
equivalence classes are countable. See \cite{MR1900547} for an introduction
to the theory of countable Borel equivalence relations. It is an important
fact that there are maximal countable Borel equivalence relations under 
$\leq_B$, and such equivalence relations are said to be
\define{universal} \cite{MR1149121}. Our first theorem is that the resource
bounded equivalence relations discussed above are universal
(see Theorem~\ref{poly_universal_2}). 

\begin{thm}\label{poly_universal}
Suppose $g(n) \geq n^2$ is a strictly increasing time-constructible function. Then any countable Borel equivalence relation $E$ such that \linebreak
$\equiv_m^{\TIME(\union_k O(g^k))} \subset E \subset
\equiv_T^{\SPACE(\union_k O(g^k))}$ is a universal countable Borel equivalence
relation.
\end{thm}

For example, polynomial time many-one equivalence and polynomial time Turing
equivalence are both universal countable Borel equivalence relations,
corresponding to the case $g(n) = n^2$. It is an open question of
Kechris~\cite{MR1791302} whether Turing equivalence is a universal countable Borel
equivalence relation.

\subsection{Ultrafilters, universality, and relativization}

Suppose $E$ is a countable Borel equivalence relation on a standard Borel
space $X$. We say that a set $A \subset X$ is \define{$E$-invariant} if
$\forall x, y \in X$ if $x E y$, then $x \in A \iff y \in A$. Of course,
the Borel $E$-invariant sets form a $\sigma$-algebra. Ultrafilters on this
$\sigma$-algebra of invariant Borel sets are important tools used in the
study of Borel reducibility of equivalence relations. Such ultrafilters
arise naturally from $E$-ergodic probability measures on $X$ and
$E$-generically ergodic Polish topologies on $X$. More recently, work
started by Simon Thomas in \cite{MR2563815} and continued in
\cite{1109.1875} and \cite{MarksLong} has used ultrafilters related to
Martin's ultrafilter on the Turing degrees to derive structural
consequences about universal equivalence relations of various kinds. For
example, in \cite{1109.1875}, these ultrafilters from computability were
used to show that if $E$ is a universal countable Borel equivalence
relation on a standard Borel space $X$, and $A \subset X$ is Borel, then
either $E \restriction A$ is universal or $E \restriction (X \setminus A)$
is universal. This answered a question of Jackson, Kechris, and Louveau
\cite{MR1900547}. These ultrafilters which preserve the Borel cardinality
of their associated equivalence relations are promising candidates for
resolving other questions from the theory of countable Borel equivalence
relations known to be resistant to measure and category techniques. See
\cite{MarksLong} for an introduction to them. 

It is an interesting problem to find other natural examples of such
ultrafilters on countable Borel equivalence relations. For example,
\cite{ThomasWilliams} has investigated whether there could be a cone
measure for the quasiorder of embeddability among countable groups. Here,
we consider such ultrafilters for resource bounded equivalence relations
from computability theory. 

The ultrafilters we consider are defined via games very similar to the game
defining Martin's ultrafilter  on the Borel Turing invariant sets (see
\cite{MR0227022}). Martin considered the game $G_A$ where two players
alternate defining longer and longer initial segments of an oracle $x$, and
at the end, player I wins if the resulting oracle is in the payoff set $A$.
Martin has shown \cite{MR0227022} that player I has a winning strategy if
and only if $A$ contains a \define{Turing cone}, that is, a set of the form
$\{x : x \geq_T y\}$ for some oracle $y$. Similarly, player II has a
winning strategy if and only if the complement of $A$ contains a Turing
cone. Hence, by Borel determinacy \cite{MR0403976}, the set of Borel
Turing-invariant $A$ that contain a Turing cone forms an ultrafilter on the
Borel Turing-invariant sets, which is known as \define{Martin's
ultrafilter} or \define{Martin measure}.

It is folklore that Martin's proof
generalizes to any quasi-order $\leq_Q$ such that given a strategy $\sigma$
in the game
and an oracle $x$ which is $\geq_Q$ an encoding of $\sigma$, then $x \equiv_Q \sigma*x$
where $\sigma * x$ is the outcome of playing $\sigma$ against $x$. 
While
this condition encompasses a large number of quasi-orders, it does not
include Turing equivalence restricted to sub-exponential time or
space bounds, since strategies in games take an exponential amount of space
to encode. 
Nor does it include many-one equivalence, or its time or space restricted
versions, since even if $x \geq_m \sigma$ where $\sigma$ is a strategy in
Martin's game, a many-one reduction relative to $x$ cannot combine
information about $x$ and this strategy to determine $\sigma * x$. Here then, we may ask to what extent
there can be an analogue of Martin's ultrafilter for these equivalence
relations. 

First, we note that standard relativization barriers show that there cannot
be an ultrafilter defined via cones for these equivalence relations. For example, relativizing the
theorem of Baker, Gill, and Solovay \cite{MR0395311} giving oracles
$x$ and $y$ relative to which 
$\PP^x = \NP^x$ and $\PP^y \neq \NP^y$ shows that both the oracles relative
to which $\PP = \NP$ and $\PP \neq \NP$ are cofinal under polynomial time
Turing reducibility $\leq_T^{\PP}$, and so
there is no cone ultrafilter on the $\equiv_T^{\PP}$-invariant sets.
Further, there can be no cone ultrafilter for $\leq_m$ since the collection
of oracles that are 
many-one equivalent to their complement are cofinal under $\leq_m$ but do not contain
a $\leq_m$-cone. 
However, it turns out that Martin's game still defines an
ultrafilter with interesting structure-preserving properties, even though
it does not have a definition via cones. 

\begin{thm}\label{ultrafilter}
  Let $E$ be a countable Borel equivalence relation as in
  Theorem~\ref{poly_universal}. Let $G_A$ be Martin's game in
  \cite{MR0227022} where the two players $I$ and $II$ alternate defining
  the elements of an oracle $x$ and $I$ wins if $x \in A$. Then the
  collection of $E$-invariant Borel sets $A$ such that player $I$ has a
  winning strategy in the game $G_A$ is an ultrafilter on the Borel
  $E$-invariant sets. Furthermore, for any $A \in U$, $E \restrict A$ is a
  universal countable Borel equivalence relation.
\end{thm}

There is some connection here with issues in computability and complexity
theory surrounding the phenomenon of relativization. In recursion theory,
proofs almost always relativize. That is, we can take most any proof in the
subject, change all Turing machines used in the proof to grant them access
to an additional oracle $x$, and the proof will remain valid. And while
early sweeping attempts to formalize this phenomenon failed, such as
Rogers' homogeneity conjecture (see \cite{MR0224462}*{\S 13.1},
\cite{MR0186554}*{\S 12},
\cite{MR543312}, and \cite{MR644748}), Martin's
ultrafilter on the Turing degrees yields a somewhat weaker explanation of
the ubiquity of this phenomenon. 
If $\psi(\textbf{x})$ is a Borel property of a Turing degrees, then either $\psi(\textbf{x})$ is true
on a cone of oracles, or $\psi(\textbf{x})$ is false on a cone of oracles.
Hence, every reasonable fact
about the Turing degrees eventually relativizes above some Turing cone. 

As we've noted above, Baker, Gill, and Solovay's result implies that the
polynomial time analogue of this result is false. However,
since Martin's cone theorem becomes true once we enlarge beyond
iterated exponential time Turing equivalence, we
can show that the connection between sets in our ultrafilter and cones
(and hence the meta-theorem preventing the existence of relativization
barriers on a cone) occurs once $E$ grows beyond iterated exponential time
Turing equivalence. That is, 
if $E$ contains $\equiv_T^{\ELEMENTARY}$, then $A
\in U$ if and only if $A$ contains a $\leq_T^{\ELEMENTARY}$-cone.
This fact tells us something interesting about
relativization barriers themselves -- that on a cone, such barriers can
always be found inside the exponential hierarchy. Suppose $\psi$ is a
$\equiv_m^{\PP}$ invariant property (such as whether two naturally defined
relativized complexity classes coincide). Suppose also
that there is a relativization barrier for $\psi$ \emph{which itself
relativizes},
so that for every $z$, there are $x \geq_m^{\PP} z$ and $y \geq_m^{\PP} z$
relative to which $\psi(x)$ is true, and $\psi(y)$ is false. Then for a
$\leq_m^{\PP}$-cone of $z$, we can find such $x$ and $y$ so that $x,y \in
\ELEMENTARY^z$. A consequence of this is that if $\psi$ is invariant on
$\equiv_T^{\ELEMENTARY}$-classes, then it must be either true or
false on some $\leq_T^{\ELEMENTARY}$-cone, and hence $\psi$ can
not admit a Baker-Gill-Solovay-type relativization barrier.

In another direction, we show that 
the set of languages $x \in \langs$ for which $\PP^x = \NP^x$ is in the
ultrafilter $U$ defined in Theorem~\ref{ultrafilter}. While
oracles relative to which $\PP =
\NP$ are often considered to be somehow ``rare'' (for example
meager \cite{BI87} and Lebesgue null \cite{MR605606}) here we have a 
natural sense in which the set of oracles relative to which $\PP = \NP$
is large--it has the maximal possible Borel cardinality. 

\section{Preliminaries}
\label{sec:prelims}

We begin by reviewing some notation and conventions. Given any set $S$, we
will often exploit the bijection via characteristic functions between its
powerset $\pow(S)$, and $2^S$, the space of functions from $S$ to $\{0,1\}$, and move freely between these two representations. We
use the notation $2^n$ for the set of finite binary strings of length $n$
and $2^{\leq n}$ for finite binary strings of length $\leq n$. The set of
all finite strings is noted $2^{< \omega}$. We use $r \concat s$ to note
the concatenation of the strings $r$ and $s$.

Define $\langs$ to be the Polish space of subsets of $\strings$. If $x, y \in \langs$, then their recursive
join is $x \join y = \{0 \concat s : s \in x\} \union \{1 \concat s: s \in
y\}$. The recursive join of finitely many elements of $\langs$ is defined
similarly. We use the notation $\comp{A}$ to denote the complement of a set
$A$. A function $g \from \N \to \N$ is said to be \define{time
constructible} if $g(n)$ is computable in $O(g(n))$-time.

In computational complexity, the issue of correctly relativizing and
modeling oracle access is a delicate and complicated matter.
See \cite{Fortnow} for a discussion of some of these issues. The theorems we
will prove, however, will be quite robust with respect to this issue, and
the only assumption we will make is the standard convention that a
computation which uses $n$ space may query only oracle strings of length at
most $n$. 

If $E$ and $F$ are equivalence relations on the spaces $X$ and $Y$, then a
function $f \from X \to Y$ is a \define{homomorphism from $E$ to $F$} if
for all $x, y \in X$ we have $x
\E y \implies f(x) \mathrel{F} f(y)$. A function $f \from X \to Y$ is a
\define{cohomomorphism from $E$ to $F$} if for all $x, y \in X$ we have
$f(x) \mathrel{F} f(y) \implies x \E y$.

\section{A universality proof}

In this section, we prove the following theorem from the introduction.

\begin{thm}\label{poly_universal_2}
Suppose $g(n) \geq n^2$ is a strictly increasing time-constructible function. 
Then any countable Borel equivalence relation $E$ such that \linebreak
$\equiv_m^{\TIME(\union_k O(g^k))} \subset E \subset
\equiv_T^{\SPACE(\union_k O(g^k))}$ is universal.
\end{thm}

\begin{proof}

We let $E_\infty$ note a universal countable Borel equivalence relation
that is generated by a continuous action of the free group on two
generators $\F_2 = \<\alpha,\beta\>$ on $\cantor$.
For example, we can use the shift action of $\F_2$ for the purpose (see
\cite{MR1149121}). We will show that
$\E$ is universal by constructing a continuous reduction $\hat{f}$ from
$E_\infty$ to $\E$.

The key to our proof is that given $\gamma \in \F_2$, we can code
$\hat{f}(\gamma
\cdot x)$ into $\hat{f}(x)$ rather sparsely so that if $|\gamma| = k$ is
the length of $\gamma$ as a reduced word, then
strings of length $n$ in $\hat{f}(\gamma \cdot x)$ are coded by strings of
length greater than $g^{k}(n)$ in $\hat{f}(x)$. From here our basic idea is
as follows: given $x, y \in \cantor$ such that $x
\mathrel{\cancel{E}_\infty} y$ and
a Turing reduction that runs in $g^k$-time, we wait till we have finite
initial segments of $x$ and $y$ witnessing that $\gamma \cdot x \neq y$ for any
$\gamma$ of length $\leq k$. Then if $n$ is large enough, we can change the
value of strings of length $n$ in $\hat{f}(x)$ without changing
$\hat{f}(y)$ restricted to strings of length $\leq g^k(n)$. This makes it
easy to diagonalize. The remaining technical wrinkle of the proof is that
we must be able to simultaneously do a lot of this sort of diagonalization.
This proof is inspired by the proof of Theorem 2.5 in \cite{1109.1875}.

We now give a precise definition of the coding we will use. Let $c:
\strings \to \strings$ be the function where $c(r) = 0^{g(|r|)}\concat 1
\concat r$ is $g(r)$ zeroes followed by a $1$ followed by $r$.
It is clear that if $x \in \langs$, then $c(x) =
\{c(r) : r \in x\}$ is many-one reducible to $x$ in $O(g)$-time, since $g$
is time constructible.

Given $f: \cantor \to \langs$, let $\hat{f}: \cantor \to \langs$ be the
unique function satisfying
\[\hat{f}(x) = f(x) \join c \left( \hat{f}(\alpha \cdot x) \join
\hat{f}(\alpha^{-1} \cdot x) \join \hat{f}(\beta \cdot x) \join
\hat{f}(\beta^{-1} \cdot
x) \right).\] 
While this definition of $\hat{f}$ is self-referential, note that if $y$ is
a language, and $s$ is a string in $y$, then $s$ is coded by a string of 
a greater length in $f(x) \join c \left( y \right)$. Hence, for each string $s$,
there is a unique end-segment $r_s$ of $s$ and a word $\gamma_s \in \F_2$ so
that $s$ is in $\hat{f}(x)$ if and only if $r_s$ is in $f(\gamma_s \cdot
x)$. So $\hat{f}$ is uniquely defined.

Now it is clear that given any $f$, the associated $\hat{f}$
is a homomorphism from $E_\infty$ to $\E$, since $\equiv_m^{\TIME(\union_k O(g^k))}
\subset E$. We claim that if $f$ is a
sufficiently generic continuous function, then $\hat{f}$ will also be a
cohomomorphism from $E_\infty$ to $\equiv_T^{\SPACE(\union_k O(g^k))}$ and hence
also a cohomomorphism from $E_\infty$ to $\E$. By generic here, we mean for the
following partial order for constructing a 
continuous (indeed, Lipschitz) function from $\cantor$ to $\langs$.
Our partial order $\P$ will consist of functions $p: 2^n \to \pow(\str{\leq n})$
such that if $m < n$ and $r_1, r_2 \in 2^n$ extend $r \in 2^m$, then
$p(r_1)$ and $p(r_2)$ agree on all strings of length $\leq m$. Given a $p : 2^n \to \pow(\str{\leq n})$ 
and $r \in 2^n$, we will often
think of $p(r)$ as a function from $\str{\leq n}$ to $2$.
If $p$ has domain $2^n$, then we will say $p$ has height $n$. If $p,p^* \in
\P$ are such that the height of $p$ is $m$ and the height of $p^*$ is $n$,
and $m \leq n$, then
say that $p^*$ extends $p$, noted $p^* \leq_{\P} p$, if for all $r^* \in
2^n$ extending $r \in 2^m$, we have that $p^*(r^*)$ extends $p(r)$ (as
functions).

If $p \in \P$, then we can define $\hat{p}$ analogously to the definition
of $\hat{f}$ above. In particular, for each $\gamma \in \F_2$, we have some
partial information about $p(\gamma \cdot r)$ based on the longest finite
initial segment of $\gamma \cdot r$ that we know (recall that the action of
$\F_2$ on $\cantor$ is continuous). Hence, given a finite string $r$,
$\hat{p}$ will map $r$ to a partial function from finite strings to $2$
that amalgamates all this partial information.

Because of our coding scheme, if $r \in \strings$, the length of $r$ is
$|r| = n$, and $\gamma \in \F_2$ is of length $k$, then whether $r \in
\hat{f}(\gamma \cdot x)$ is canonically coded into $\hat{f}(x)$ at some string
of length greater than $g^k(n)$.

Let $p_0$ be the condition of height $1$ where $p_0(r) = \emptyset$ for all
$r$. Hence, if $f: \cantor \to \langs$ extends $p_0$, then every string in $f(x)$
must have length $\geq 2$. For convenience, the generic function we
construct will extend $p_0$. 
Note that since $g(n) \geq n^2$, if $f$ is $O(g^{k-1})$ for some $k$, then
$f(n) \leq g(n)^k$ for sufficiently large $n$. 

Suppose we are given a Turing reduction $\varphi_e$ that runs in $g^k$-space, and $r,s \in 2^m$ such that $\gamma \cdot r$ is incompatible
with $s$ for all $\gamma \in \F_2$ where $|\gamma| \leq k$.
Let $D_{r,s,k,e}$ be the
set of $p$ of height $\geq m$ such that if $p$ has height $i$, then there
exists an $n \geq m$ so that if $t \in 2^i$, then $\hat{p}(t)$ is defined on all
strings of length $\leq g^k(n)$, and for all $r^*, s^* \in 2^i$ extending
$r$ and $s$, we have that $\varphi_e$ is not a $g^k$-time reduction of
$\hat{p}(s^*)$ to $\hat{p}(r^*)$ as witnessed by some string $t$ of length
$n$, so that $\varphi_e(\hat{p}(r^*))(t) \neq \hat{p}(s^*)(t)$.
We claim that $D_{r,s,k,e}$ is dense below $p_0$. The theorem will follow
from this fact.

Given any $p \leq_{\P} p_0$ where $p$ has height $j$, we must construct an
extension $p^*$ of $p$ so that $p^* \in D_{r,s,k,e}$. Fix an $n$ and an $i$
such that $i \gg n \gg j$ with the precise bounds we give below.
Define $q$ of height $i$ where for all $r$, $q(r)$
contains only the strings in $p(r \restrict j)$. We require that $2^{j+1}
4^{k+1} \log_2(n)^2 < n$, and that $i \gg n$ is large enough so that
for all $r \in 2^i$, $\hat{q}(r)$ is defined on all strings of
length $\leq g^k(n)$. Such an $i$ exists by the continuity of the action of
$\F_2$ on $\cantor$.

Given any $r^* \in 2^i$ we compute an upper bound on how many elements
$\hat{q}(r^*)$ could have of length $\leq g^k(n)$. Clearly every $q(t)$ has
less than $2^{j+1}$ elements. It will be enough to establish an upper bound
on the number of words $\gamma \in \F_2$ so that some element of $q(\gamma
\cdot r^*)$ is coded into $\hat{q}(r)$ via a string of length $\leq
g^k(n)$.

Since $q$ extends $p_0$, any element of any $q(\gamma \cdot r^*)$ must be a
string of length $\geq 2$. Now if $\gamma \in \F_2$ is such that some string of
length $\geq 2$ in $q(\gamma \cdot r^*)$ is coded into $\hat{q}(r^*)$ below
$g^k(n)$, then it must be that $g^{|\gamma|}(2) \leq g^k(n)$. Since $g(n)
\geq n^2$, we have that $g^{i}(2) \geq 2^{2^{i}}$. Hence, 
$|\gamma| \leq 
k + \log_2\log_2(n)$. This is because otherwise 
\[g^{|\gamma|}(2) \geq g^{k + \lceil \log_2
\log_2(n) \rceil}(2) = g^{k} \circ g^{\lceil \log_2 \log_2(n)\rceil}(2)
\geq g^k(n)\]
and furthermore $g^{|\gamma|}(2) > g^{k + \lceil \log_2
\log_2(n) \rceil}$ if $\log_2 \log_2 (n)$ is an integer, and $g^{k + \lceil \log_2 \log_2(n)\rceil}(2) >
g^k(n)$ if $\log_2 \log_2(n)$ is not an integer, contradicting $g^{|\gamma|}(2) \leq g^k(n)$
either way. 
Thus, there are most $4^{k+1} \log_2(n)^2$ such group elements $\gamma$
since there are at most 
$4^{l+1}$ words of $\F_2$ of length $\leq l$. 

Since each $q(\gamma \cdot r^*)$ must have less than $2^{j+1}$ elements,
$\hat{q}(r^*)$ contains at most $2^{j+1} 4^{k+1} \log_2(n)^2$ strings of length $\leq
g^k(n)$. Let $S$ be this set of all possible strings in $\hat{q}(r^*)$ of
length $\leq g^k(n)$, so
we have picked $n$ so that $|S| < n$. Note that $S$ does not depend on
$r^*$.

We see now that amongst all of the $r^*$ extending $r$, there are 
$\leq 2^{n-1}$ possibilities for what any $\hat{q}(r^*)$ restricted to
strings of length $\leq g^k(n)$ could be: they are all
elements of $\pow(S)$. Let $u_0, u_1, \ldots$ be a listing of the elements
of $\pow(S)$. Recall that based on our definition of recursive join and
$\hat{p}$, strings of length $n-1$ in $p(s^*)$ are coded into
$\hat{p}(s^*)$ using strings of length $n$ that begin with $0$. Define $p^*$
to be equal to $q$ except on extensions $s^*$ of $s$. There, if $\sigma$ is
the
$l$th string of length $n-1$, then put $\sigma \in p^*(s^*)$ if and only if
$\sigma$ is not accepted by $\varphi_e$ run relative to $u_l$. We then have
that $p^* \in D_{r,s,k,e}$ as desired. Recall here that $\gamma \cdot r^*$ is incompatible
with $s^*$ for all $\gamma \in \F_2$ where $|\gamma| \leq k$, since $r^*$
and $s^*$ extend $r$ and $s$ respectively.
\end{proof}

\section{Relativization barriers and largeness notions for sets of
polynomial time degrees}

Martin's ultrafilter is defined by a game where two players alternate
defining the bits of a real, as follows: 

\begin{defn}\label{def_ultrafilter}
  Let $U$ be the collection of $\equiv_m^{\PP}$-invariant Borel sets $A \subset \langs$ such that player I
  has a winning strategy in the following game $G_A$. Players I and II alternate
  defining which strings are in a language $x$ where on turn $n$, player I
  decides membership in $x$ for all strings of length $n$ beginning with a
  $0$, and on turn $n$, player II decides membership in $x$ for all strings
  of length $n$ beginning with a $1$, so $x$ is the recursive join of the
  languages played by I and II. Then player I wins the game if $x \in A$.
\end{defn}

We now show that $U$ is a $\sigma$-complete ultrafilter for
$\equiv_m^{\PP}$-invariant sets. 

\begin{prop}
  $U$ is a $\sigma$-complete ultrafilter on the $\equiv_m^{\PP}$-invariant
  Borel
  sets. 
\end{prop}
\begin{proof}
  Because $A$ is $\equiv_m^\PP$-invariant, it is easy to check that for any
  definition of a recursive join $\oplus$ such that the map $(x,y) \mapsto
  x \oplus y$ is polynomial time many-one equivalent to our canonical
  definition in Section~\ref{sec:prelims}, the winner of the game does not
  change. Hence, we can see that $A \in U$ iff $\comp{A} \notin U$ by
  simply switching the roles of players I and II in $G_A$. 
  
  Similarly, suppose we are given countably many strategies $\sigma_0, \sigma_1, \ldots$
  for player I in the games $G_{A_0}, G_{A_1}, \ldots$ respectively. By changing the definition of recursive join, we may assume that in the game
  $G_{A_i}$, player I decides membership in $x$ for all 
  string which begin with the string $0^{i+1} \concat 1$ of $i+1$ zeroes
  followed by a one, and player II decides membership in $x$ for all other
  strings. Now we can use the strategies $\sigma_0, \sigma_1,
  \ldots$ simultaneously to give a strategy for player I in the game where
  player I decides $x$ on all strings beginning with $0$. This
  combined strategy shows that player I wins  the game with payoff set
  $\biginters_i A_i$. Hence, $U$ is $\sigma$-complete.
\end{proof}

Next, an easy extension of Theorem~\ref{poly_universal_2} shows the
following:

\begin{thm}\label{ultrafilter_2}
  Let $E$ be a countable Borel equivalence relation as in
  Theorem~\ref{poly_universal_2}. Then Definition~\ref{def_ultrafilter}
  defines an ultrafilter on the Borel $E$-invariant sets such that for any
  $A \in U$, $E \restrict A$ is a universal countable Borel equivalence
  relation.
\end{thm}
\begin{proof}
  Given a winning strategy $\sigma$ for player I in
  the game $G_A$, and any $y \in \langs$, let $\sigma * y$ be the outcome of
  playing the strategy $\sigma$ against player II playing $y$. Then 
  if we replace the definition of $\hat{f}$ in the proof
  with 
\[\hat{f}(x) = \sigma * \left( f(x) \join c \left( \hat{f}(\alpha \cdot x) \join
\hat{f}(\alpha^{-1} \cdot x) \join \hat{f}(\beta \cdot x) \join
\hat{f}(\beta^{-1} \cdot
x) \right) \right) \] 
then $\hat{f}$ remains a homomorphism of $E_\infty$ into $E$. This is
because for all
$x$, $\hat{f}(\alpha \cdot x)$, $\hat{f}(\alpha^{-1} \cdot x)$,
$\hat{f}(\beta \cdot x)$, and $\hat{f}(\beta^{-1} \cdot x)$ are coded in
$\hat{f}(x)$ so that they are $O(g)$-time computable from $\hat{f}(x)$.
Hence, by induction, if $\gamma$ is a word of length $k$, then
$\hat{f}(\gamma \cdot x)$ is $O(g^k)$-time computable from $\hat{f}(x)$.

Furthermore,
a trivial modification of the proof of Theorem~\ref{poly_universal_2} shows that with
this new definition of $\hat{f}$ and the associated $\hat{p}$, the sets
$D_{r,s,k,e}$ are still dense and hence 
if $f$
is a sufficiently generic continuous function, then $\hat{f}$ will be a
Borel
reduction. 
Finally, $\hat{f}$ must be
an embedding whose range is contained in the set $A$, since $\sigma$ is a
winning strategy for player I in $G_A$.
\end{proof}

Now an easy argument essentially repeating Gill, Baker, and Solovay's proof
shows that if $A = \{x : \PP^x = \NP^x\}$, then there is a winning strategy
for player I in the game $G_A$; player I simply ensures that the outcome of
the game is a language that is polynomial time equivalent to its $\PSPACE$
completion. Thus, we have the following corollary

\begin{cor}
  Let $A = \{x : \PP^x = \NP^x\}$. Then $\equiv_m^{\PP} \restrict A$ and
  $\equiv_T^{\PP} \restrict A$ are universal countable Borel equivalence
  relations.
\end{cor}

We finish by noting that the ultrafilter $U$ in
Definition~\ref{def_ultrafilter} can be combined with
Theorem~\ref{ultrafilter_2} to provide an alternate 
way of
proving many of the results in section 3 of \cite{1109.1875};
one simply uses them to replace Martin's measure and arithmetic
equivalence in that paper. 

\bibliography{references}

\end{document}